\documentclass{amsart}

\usepackage{amssymb}

\usepackage{tikz}
\usepackage{graphicx}

\theoremstyle{plain}
\newtheorem{theorem}{Theorem}
\newtheorem{lemma}[theorem]{Lemma}
\newtheorem{proposition}[theorem]{Proposition}
\newtheorem*{maintheorem}{Theorem~\ref{theorem:mainthm}}

\theoremstyle{definition}
\newtheorem{definition}[theorem]{Definition}
\newtheorem{example}[theorem]{Example}

\theoremstyle{remark}
\newtheorem{remark}[theorem]{Remark}
\newtheorem{claim}[theorem]{Claim}
\newtheorem*{acknowledgments}{Acknowledgments}

\tikzset{myarrow/.style={->,>=stealth,line width=0.5pt}}

\newcommand{\su}[1]{\mathcal{S}_{#1}}
\newcommand{\sv}[3]{\mathcal{S}_{#1}^{#2, #3}}
\newcommand{\xu}[1]{\mathcal{X}_{#1}}
\newcommand{\xv}[3]{\mathcal{X}_{#1}^{#2, #3}}
\newcommand{\sxu}{\boldsymbol{\Sigma}}
\newcommand{\sxv}[2]{\boldsymbol{\Sigma}^{#1, #2}}
\newcommand{\sxus}[1]{\boldsymbol{\Sigma}_{#1}}
\newcommand{\sxvs}[3]{\boldsymbol{\Sigma}_{#1}^{#2, #3}}

\DeclareMathOperator{\sgn}{sgn}
\DeclareMathOperator{\Gal}{Gal}

\begin{document}

\title{Simultaneously preperiodic integers for quadratic polynomials}

\author{Valentin Huguin}
\address{Institut de Mathématiques de Toulouse, UMR 5219, Université de Toulouse, CNRS, UPS, F-31062 Toulouse Cedex 9, France}
\email{valentin.huguin@math.univ-toulouse.fr}

\subjclass[2010]{Primary 37P05; Secondary 37F45, 37P35}

\begin{abstract}
In this article, we study the set of parameters $c \in \mathbb{C}$ for which two given complex numbers $a$ and $b$ are simultaneously preperiodic for the quadratic polynomial $f_{c}(z) = z^{2} +c$. Combining complex-analytic and arithmetic arguments, Baker and DeMarco showed that this set of parameters is infinite if and only if $a^{2} = b^{2}$. Recently, Buff answered a question of theirs, proving that the set of parameters $c \in \mathbb{C}$ for which both $0$ and $1$ are preperiodic for $f_{c}$ is equal to $\lbrace -2, -1, 0 \rbrace$. Following his approach, we complete the description of these sets when $a$ and $b$ are two given integers with $\lvert a \rvert \neq \lvert b \rvert$.
\end{abstract}

\maketitle

\section{Introduction}

For $c \in \mathbb{C}$, let $f_{c} \colon \mathbb{C} \rightarrow \mathbb{C}$ be the complex quadratic map \[ f_{c} \colon z \mapsto z^{2} +c \, \text{.} \]

Given a point $z \in \mathbb{C}$, we study the sequence $\left( f_{c}^{\circ n}(z) \right)_{n \geq 0}$ of iterates of $f_{c}$ at $z$. The set $\left\lbrace f_{c}^{\circ n}(z) : n \geq 0 \right\rbrace$ is called the \emph{forward orbit} of $z$ under $f_{c}$.

The point $z$ is said to be \emph{periodic} for $f_{c}$ if there exists an integer $p \geq 1$ such that $f_{c}^{\circ p}(z) = z$. The least such integer $p$ is called the \emph{period} of $z$. The point $z$ is said to be \emph{preperiodic} for $f_{c}$ if its forward orbit is finite or, equivalently, if there is an integer $k \geq 0$ such that $f_{c}^{\circ k}(z)$ is periodic for $f_{c}$. The smallest integer $k$ with this property is called the \emph{preperiod} of $z$.

\begin{definition}
For $a \in \mathbb{C}$, let $\su{a}$ be the set defined by \[ \su{a} = \left\lbrace c \in \mathbb{C} : a \text{ is preperiodic for } f_{c} \right\rbrace \, \text{.} \]
\end{definition}

In this paper, we wish to examine these sets of parameters.

For $n \geq 0$, let $F_{n} \in \mathbb{Z}[c, z]$ be the polynomial given by \[ F_{n}(c, z) = f_{c}^{\circ n}(z) \, \text{.} \] The sequence $\left( F_{n} \right)_{n \geq 0}$ satisfies $F_{0}(c, z) = z$ and the recursion formulas \[ F_{n}(c, z) = F_{n -1}\left( c, z^{2} +c \right) = F_{n -1}(c, z)^{2} +c \quad \text{for} \quad n \geq 1 \, \text{.} \] In particular, when $n \geq 1$, the polynomial $F_{n}$ is monic in $c$ of degree $2^{n -1}$ and monic in $z$ of degree $2^{n}$.

Now, given a point $a \in \mathbb{C}$, define~-- for $k \geq 0$ and $p \geq 1$~-- the set \[ \sv{a}{k}{p} = \left\lbrace c \in \mathbb{C} : F_{k +p}(c, a) = F_{k}(c, a) \right\rbrace \, \text{.} \] For all $k \geq 0$ and $p \geq 1$, the set $\sv{a}{k}{p}$ contains at most $2^{k +p -1}$ elements and consists of the parameters $c \in \mathbb{C}$ for which the point $a$ is preperiodic for $f_{c}$ with preperiod less than or equal to $k$ and period dividing $p$.

In particular, it follows that the set \[ \su{a} = \bigcup_{k \geq 0, \, p \geq 1} \sv{a}{k}{p} \] is countable.

\begin{proposition}
\label{proposition:infinite}
For every $a \in \mathbb{C}$, the set $\su{a}$ is infinite.
\end{proposition}

\begin{proof}
To obtain a contradiction, suppose that $\su{a}$ contains finitely many elements. Then, since the sequence $\left( \sv{a}{n}{1} \right)_{n \geq 0}$ is increasing, there exists an integer $N \geq 0$ such that $\sv{a}{n +1}{1} = \sv{a}{n}{1}$ for all $n \geq N$. Now, note that, for every $n \geq 0$, we have \[ F_{n +2}(c, a) -F_{n +1}(c, a) = \left( F_{n +1}(c, a) -F_{n}(c, a) \right) \left( F_{n +1}(c, a) +F_{n}(c, a) \right) \, \text{.} \] It follows that, if $n \geq N$ and $\gamma$ is a root of the polynomial $F_{n +1}(c, a) +F_{n}(c, a)$, then \[ F_{n +1}(\gamma, a) -F_{n}(\gamma, a) = F_{n +1}(\gamma, a) +F_{n}(\gamma, a) = 0 \, \text{,} \] and hence $F_{n +1}(\gamma, a) = F_{n}(\gamma, a) = 0$, which yields $\gamma = 0$. Therefore, we have $F_{n}(0, a) = 0$ and $F_{n +1}(c, a) +F_{n}(c, a) = c^{2^{n}}$ for all $n \geq N$. In particular, we get \[ \frac{\partial\left( F_{N +2} +F_{N +1} \right)}{\partial c}(0, a) = 2 \frac{\partial F_{N +1}}{\partial c}(0, a) F_{N +1}(0, a) +2 \frac{\partial F_{N}}{\partial c}(0, a) F_{N}(0, a) +2 = 2 \, \text{,} \] which contradicts the fact that $F_{N +2}(c, a) +F_{N +1}(c, a) = c^{2^{N +1}}$.
\end{proof}

\begin{remark}
Note that, if $a \in \mathbb{C}$, then $f_{c}(a) = f_{c}(-a)$ for all $c \in \mathbb{C}$. Consequently, we have $\su{a} = \su{-a}$ and $\sv{a}{k}{p} = \sv{-a}{k}{p}$ for all $k \geq 1$ and $p \geq 1$.
\end{remark}

\begin{example}
Assume that $a \in \mathbb{C}$. Then (see Figure~\ref{figure:genexu}) we have 
\begin{align*}
\sv{a}{0}{1} & = \left\lbrace -a^{2} +a \right\rbrace \, \text{,}\\
\sv{a}{1}{1} & = \left\lbrace -a^{2} -a, -a^{2} +a \right\rbrace \, \text{,}\\
\sv{a}{0}{2} & = \left\lbrace -a^{2} -a -1, -a^{2} +a \right\rbrace \, \text{,}\\
\sv{a}{1}{2} & = \left\lbrace -a^{2} -a -1, -a^{2} -a, -a^{2} +a -1, -a^{2} +a \right\rbrace \, \text{.}
\end{align*}
\end{example}

\begin{figure}
\fbox{\begin{tikzpicture}
\begin{scope}
\node (Z0) at (2,0) {$a$};
\node[right] (C) at (5,0) {for $c = -a^{2} +a$,};
\draw[myarrow] (Z0) to[out=45,in=0,loop] node[above right]{$f_{c}$} (Z0);
\end{scope}
\begin{scope}[yshift=-1.5cm]
\node (Z0) at (1,0) {$a$};
\node (Z1) at (3,0) {$-a$};
\node[right] (C) at (5,0) {for $c = -a^{2} -a$,};
\draw[myarrow] (Z0) to[bend left] node[above]{$f_{c}$} (Z1);
\draw[myarrow] (Z1) to[out=45,in=0,loop] node[above right]{$f_{c}$} (Z1);
\end{scope}
\begin{scope}[yshift=-3cm]
\node (Z0) at (1,0) {$a$};
\node (Z1) at (3,0) {$-a -1$};
\node[right] (C) at (5,0) {for $c = -a^{2} -a -1$,};
\draw[myarrow] (Z0) to[bend left] node[above]{$f_{c}$} (Z1);
\draw[myarrow] (Z1) to[bend left] node[below]{$f_{c}$} (Z0);
\end{scope}
\begin{scope}[yshift=-5cm]
\node (Z0) at (0,0) {$a$};
\node (Z1) at (2,0) {$a -1$};
\node (Z2) at (4,0) {$-a$};
\node[right] (C) at (5,0) {for $c = -a^{2} +a -1$.};
\draw[myarrow] (Z0) to[bend left] node[above]{$f_{c}$} (Z1);
\draw[myarrow] (Z1) to[bend left] node[above]{$f_{c}$} (Z2);
\draw[myarrow] (Z2) to[bend left] node[below]{$f_{c}$} (Z1);
\end{scope}
\end{tikzpicture}}
\caption{Some parameters $c \in \mathbb{C}$ for which a given complex number $a$ is preperiodic for $f_{c}$.}
\label{figure:genexu}
\end{figure}
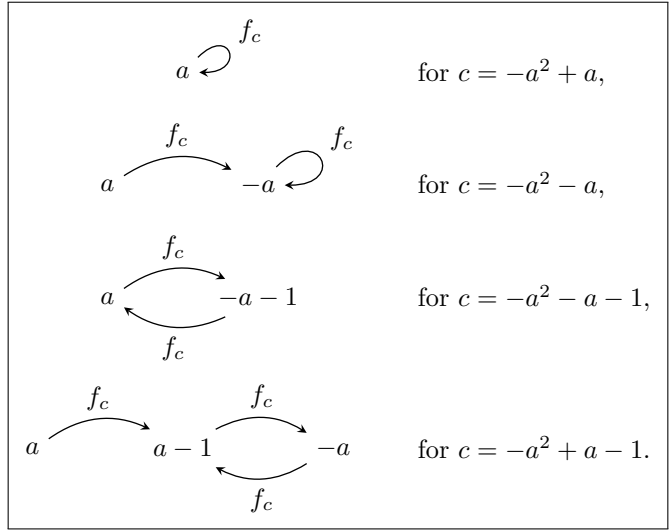

Here, the problem we are interested in is the description of the sets $\su{a} \cap \su{b}$ when $a$ and $b$ are two given complex numbers.

\begin{example}
\label{example:genexv}
Suppose that $a \in \mathbb{C}$. Then (see Figure~\ref{figure:genexv}) we have \[ -a^{2} -a -1 = -(a +1)^{2} +(a +1) -1 \in \sv{a}{0}{2} \cap \sv{a +1}{1}{2} \] and \[ -a^{2} -a = -(a +1)^{2} +(a +1) \in \sv{a}{1}{1} \cap \sv{a +1}{0}{1} \, \text{.} \]
\end{example}

\begin{figure}
\fbox{\begin{tikzpicture}
\begin{scope}
\node (Z0) at (1,0) {$a +1$};
\node (Z1) at (3,0) {$a$};
\node (Z2) at (5,0) {$-a -1$};
\node[right] (C) at (7,0) {for $c = -a^{2} -a -1$,};
\draw[myarrow] (Z0) to[bend left] node[above]{$f_{c}$} (Z1);
\draw[myarrow] (Z1) to[bend left] node[above]{$f_{c}$} (Z2);
\draw[myarrow] (Z2) to[bend left] node[below]{$f_{c}$} (Z1);
\end{scope}
\begin{scope}[yshift=-2cm]
\node (Z0) at (0,0) {$a$};
\node (Z1) at (2,0) {$-a$};
\node (AND) at (4,0) {and};
\node (Z2) at (5,0) {$a +1$};
\node[right] (C) at (7,0) {for $c = -a^{2} -a$.};
\draw[myarrow] (Z0) to[bend left] node[above]{$f_{c}$} (Z1);
\draw[myarrow] (Z1) to[out=45,in=0,loop] node[above right]{$f_{c}$} (Z1);
\draw[myarrow] (Z2) to[out=45,in=0,loop] node[above right]{$f_{c}$} (Z2);
\end{scope}
\end{tikzpicture}}
\caption{Two parameters $c \in \mathbb{C}$ for which $a$ and $a +1$ are simultaneously preperiodic for $f_{c}$ when $a$ is a given complex number.}
\label{figure:genexv}
\end{figure}
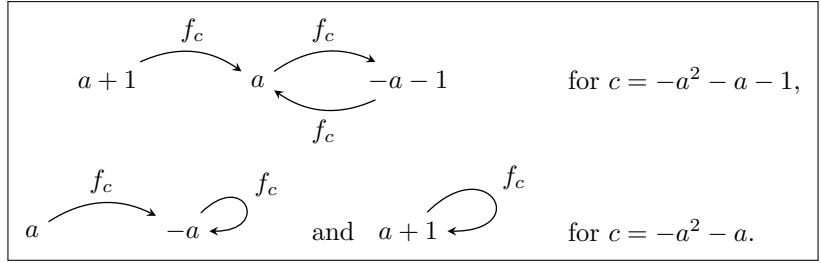

\begin{example}
\label{example:partexv}
We have $-2 \in \sv{0}{2}{1} \cap \sv{1}{1}{1}$, $-1 \in \sv{0}{0}{2} \cap \sv{1}{1}{2}$ and $0 \in \sv{0}{0}{1} \cap \sv{1}{0}{1}$ (see Figure~\ref{figure:partexv}).
\end{example}

\begin{figure}
\fbox{\begin{tikzpicture}
\begin{scope}
\node (Z0) at (0,0) {$0$};
\node (Z1) at (2,0) {$-2$};
\node (Z2) at (4,0) {$2$};
\node (AND) at (6,0) {and};
\node (Z3) at (7,0) {$1$};
\node (Z4) at (9,0) {$-1$};
\draw[myarrow] (Z0) to[bend left] node[above]{$f_{-2}$} (Z1);
\draw[myarrow] (Z1) to[bend left] node[above]{$f_{-2}$} (Z2);
\draw[myarrow] (Z2) to[out=45,in=0,loop] node[above right]{$f_{-2}$} (Z2);
\draw[myarrow] (Z3) to[bend left] node[above]{$f_{-2}$} (Z4);
\draw[myarrow] (Z4) to[out=45,in=0,loop] node[above right]{$f_{-2}$} (Z4);
\end{scope}
\begin{scope}[yshift=-1.5cm]
\node (Z0) at (3,0) {$1$};
\node (Z1) at (5,0) {$0$};
\node (Z2) at (7,0) {$-1$};
\draw[myarrow] (Z0) to[bend left] node[above]{$f_{-1}$} (Z1);
\draw[myarrow] (Z1) to[bend left] node[above]{$f_{-1}$} (Z2);
\draw[myarrow] (Z2) to[bend left] node[below]{$f_{-1}$} (Z1);
\end{scope}
\begin{scope}[yshift=-3.5cm]
\node (Z0) at (3,0) {$0$};
\node (AND) at (5,0) {and};
\node (Z1) at (6,0) {$1$};
\draw[myarrow] (Z0) to[out=45,in=0,loop] node[above right]{$f_{0}$} (Z0);
\draw[myarrow] (Z1) to[out=45,in=0,loop] node[above right]{$f_{0}$} (Z1);
\end{scope}
\end{tikzpicture}}
\caption{Three parameters $c \in \mathbb{C}$ for which both $0$ and $1$ are preperiodic for $f_{c}$.}
\label{figure:partexv}
\end{figure}
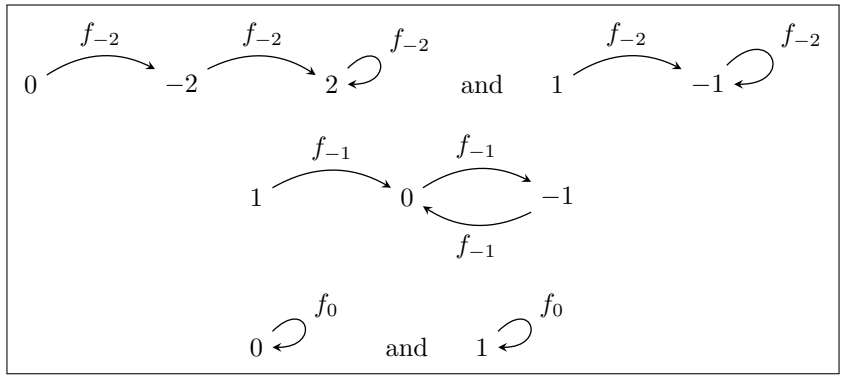

Since the sets $\su{a}$ are countably infinite (see Proposition~\ref{proposition:infinite}), we may wonder whether the sets $\su{a} \cap \su{b}$ are infinite. This question was answered by Baker and DeMarco in~\cite{BDM2011}. Using potential theory and an equidistribution result for points of small height with respect to an adelic height function, they proved that the set $\su{a} \cap \su{b}$ is infinite if and only if $a^{2} = b^{2}$.

As they pointed out, their proof is not effective and does not provide any estimate on the cardinality of these sets when they are finite. In their article, Baker and DeMarco conjectured that $-2$, $-1$ and $0$ were the only parameters $c \in \mathbb{C}$ for which $0$ and $1$ are simultaneously preperiodic for $f_{c}$ (see Example~\ref{example:partexv}). Using localization properties of the set of parameters $c \in \mathbb{C}$ for which both $0$ and $1$ have bounded forward orbit under $f_{c}$ and the fact that $0$ is the only parameter $c \in \mathbb{C}$ that is contained in the main cardioid of the Mandelbrot set and for which $0$ is preperiodic for $f_{c}$, Buff gave an elementary proof of their conjecture in~\cite{B2018}.

Following his approach, we complete the description of the sets $\su{a} \cap \su{b}$ when $a$ and $b$ are two given integers with $\lvert a \rvert \neq \lvert b \rvert$. More precisely, we prove the following theorem, which asserts that Example~\ref{example:genexv} and Example~\ref{example:partexv} present all the parameters $c \in \mathbb{C}$ for which two given distinct and non-opposite integers are simultaneously preperiodic for the polynomial $f_{c}$:

\begin{theorem}
\label{theorem:mainthm}
Assume that $a$ and $b$ are two integers with $\lvert b \rvert > \lvert a \rvert$. Then 
\begin{itemize}
\item either $a = 0$, $\lvert b \rvert = 1$ and $\su{a} \cap \su{b} = \lbrace -2, -1, 0 \rbrace$,
\item or $a = 0$, $\lvert b \rvert = 2$ and $\su{a} \cap \su{b} = \lbrace -2 \rbrace$,
\item or $\lvert a \rvert \geq 1$, $\lvert b \rvert = \lvert a \rvert +1$ and $\su{a} \cap \su{b} = \left\lbrace -a^{2} -\lvert a \rvert -1, -a^{2} -\lvert a \rvert \right\rbrace$,
\item or $\lvert b \rvert > \max\left\lbrace 2, \lvert a \rvert +1 \right\rbrace$ and $\su{a} \cap \su{b} = \varnothing$.
\end{itemize}
\end{theorem}

Our proof is elementary and uses only basic analytic and arithmetic arguments. In particular, the reader does not need to be familiar with complex dynamics.

In Section~\ref{section:dynaspace}, we reprove some well-known results on the dynamics of the polynomials $f_{c}$. In Section~\ref{section:paraspace}, we go back to the study of the parameter space and give a proof of Theorem~\ref{theorem:mainthm}.

\begin{acknowledgments}
The author would like to thank his Ph.D. advisors, Xavier Buff and Jasmin Raissy, for helpful discussions without which this paper would not exist.
\end{acknowledgments}

\section{The dynamics of the quadratic polynomials}
\label{section:dynaspace}

We shall investigate here the dynamics of the quadratic maps $f_{c} \colon \mathbb{C} \rightarrow \mathbb{C}$.

Given a parameter $c \in \mathbb{C}$, let $\xu{c}$ be the set \[ \xu{c} = \left\lbrace z \in \mathbb{C} : z \text{ is preperiodic for } f_{c} \right\rbrace \, \text{,} \] and, for $k \geq 0$ and $p \geq 1$, let $\xv{c}{k}{p}$ be the set \[ \xv{c}{k}{p} = \left\lbrace z \in \mathbb{C} : F_{k +p}(c, z) = F_{k}(c, z) \right\rbrace \, \text{.} \] For all $k \geq 0$ and $p \geq 1$, the set $\xv{c}{k}{p}$ contains at most $2^{k +p}$ elements, is invariant under $f_{c}$ and consists of the preperiodic points for $f_{c}$ with preperiod less than or equal to $k$ and period dividing $p$. In particular, we have \[ \xu{c} = \bigcup_{k \geq 0, \, p \geq 1} \xv{c}{k}{p} \, \text{.} \] Moreover, the set $\xu{c}$ is completely invariant under $f_{c}$~-- that is, for every $z \in \mathbb{C}$, $f_{c}(z) \in \xu{c}$ if and only if $z \in \xu{c}$.

\begin{remark}
Note that, if $c \in \mathbb{C}$, then $f_{c}(z) = f_{c}(-z)$ for all $z \in \mathbb{C}$. Therefore, the sets $\xu{c}$ and $\xv{c}{k}{p}$, with $k \geq 1$ and $p \geq 1$, are symmetric with respect to the origin.
\end{remark}

\begin{proposition}
\label{proposition:dynaloc}
For every $c \in \mathbb{C}$, we have \[ \xu{c} \subset \bigcap_{n \geq 0} \left\lbrace z \in \mathbb{C} : \left\lvert f_{c}^{\circ n}(z) \right\rvert \leq \rho_{c} \right\rbrace \, \text{,} \] where $\rho_{c} = \frac{1 +\sqrt{1 +4 \lvert c \rvert}}{2}$.
\end{proposition}

\begin{proof}
For every $z \in \mathbb{C}$, we have $\left\lvert f_{c}(z) \right\rvert \geq \lvert z \rvert^{2} -\lvert c \rvert$, and $\lvert z \rvert^{2} -\lvert c \rvert > \lvert z \rvert$ if and only if $\lvert z \rvert > \rho_{c}$. It follows by induction that, if $z \in \mathbb{C}$ satisfies $\lvert z \rvert > \rho_{c}$, then $\left\lvert f_{c}^{\circ (k +p)}(z) \right\rvert > \left\lvert f_{c}^{\circ k}(z) \right\rvert$ for all $k \geq 0$ and $p \geq 1$, and hence $z$ is not preperiodic for $f_{c}$. As the set $\xu{c}$ is invariant under $f_{c}$, this completes the proof of the proposition.
\end{proof}

Now, let us study the dynamics of the polynomial $f_{c}$ when $c$ is a real parameter. Suppose that $c \in \left( -\infty, \frac{1}{4} \right]$. Then the map $f_{c} \colon \mathbb{R} \rightarrow \mathbb{R}$ is even and strictly increasing on $\mathbb{R}_{\geq 0}$, has two fixed points $\alpha_{c} \leq \beta_{c}$~-- with equality if and only if $c = \frac{1}{4}$~-- given by \[ \alpha_{c} = \frac{1 -\sqrt{1 -4 c}}{2} \quad \text{and} \quad \beta_{c} = \frac{1 +\sqrt{1 -4 c}}{2} \] and satisfies $f_{c}(z) > z$ for all $z \in \left( \beta_{c}, +\infty \right)$. In particular, we have \[ f_{c}\left( \left[ -\beta_{c}, \beta_{c} \right] \right) = \left[ c, \beta_{c} \right] \] and the sequence $\left( f_{c}^{\circ n}(z) \right)_{n \geq 0}$ diverges to $+\infty$ for all $z \in \left( -\infty, -\beta_{c} \right) \cup \left( \beta_{c}, +\infty \right)$.

It follows that, if $c \in \left[ -2, \frac{1}{4} \right]$, then \[ f_{c}\left( \left[ -\beta_{c}, \beta_{c} \right] \right) \subset \left[ -\beta_{c}, \beta_{c} \right] \, \text{,} \] and hence, for every $z \in \mathbb{R}$, the point $z$ has bounded forward orbit under $f_{c}$ if and only if $z \in \left[ -\beta_{c}, \beta_{c} \right]$.

\begin{remark}
Note that, for every $c \in \mathbb{C}$, we have $\rho_{c} = \beta_{-\lvert c \rvert}$.
\end{remark}

Let us examine more thoroughly the dynamics of the map $f_{c}$ when $c \in (-\infty, -2]$. It is related to the dynamics of the shift map in the space of sign sequences.

Let $\sigma \colon \lbrace -1, 1 \rbrace^{\mathbb{Z}_{\geq 0}} \rightarrow \lbrace -1, 1 \rbrace^{\mathbb{Z}_{\geq 0}}$ denote the \emph{shift map}, which sends any sequence $\boldsymbol{\varepsilon} = \left( \epsilon_{n} \right)_{n \geq 0}$ of $\pm 1$ to the sequence $\left( \epsilon_{n +1} \right)_{n \geq 0}$.

A sign sequence $\boldsymbol{\varepsilon}$ is said to be \emph{periodic} with \emph{period} $p \geq 1$ if $\sigma^{\circ p}(\boldsymbol{\varepsilon}) = \boldsymbol{\varepsilon}$ and $p$ is the least such integer. The sequence $\boldsymbol{\varepsilon}$ is said to be \emph{preperiodic} with \emph{preperiod} $k \geq 0$ if the sequence $\sigma^{\circ k}(\boldsymbol{\varepsilon})$ is periodic and $k$ is minimal with this property.

For $k \geq 0$ and $p \geq 1$, define \[ \sxv{k}{p} = \left\lbrace \boldsymbol{\varepsilon} \in \lbrace -1, 1 \rbrace^{\mathbb{Z}_{\geq 0}} : \sigma^{\circ (k +p)}(\boldsymbol{\varepsilon}) = \sigma^{\circ k}(\boldsymbol{\varepsilon}) \right\rbrace \] to be the set of all preperiodic sign sequences with preperiod less than or equal to $k$ and period dividing $p$, and define \[ \sxu = \bigcup_{k \geq 0, \, p \geq 1} \sxv{k}{p} \] to be the collection of all preperiodic sign sequences. For all $k \geq 0$ and $p \geq 1$, the set $\sxv{k}{p}$ contains exactly $2^{k +p}$ elements~-- each of them being completely determined by the choice of its first $k +p$ terms~-- and is invariant under the shift map. Moreover, the set $\sxu$ is completely invariant under the shift map~-- that is, any sign sequence $\boldsymbol{\varepsilon}$ is preperiodic if and only if the sequence $\sigma(\boldsymbol{\varepsilon})$ is preperiodic.

\begin{theorem}
\label{theorem:dynashift}
For every $c \in (-\infty, -2]$, there exists a unique map \[ \psi_{c} \colon \sxu \rightarrow \mathbb{R} \] that makes the diagram below commute and satisfies $\epsilon_{0} \psi_{c}(\boldsymbol{\varepsilon}) \geq 0$ for all $\boldsymbol{\varepsilon} \in \sxu$.
\begin{center}
\begin{tikzpicture}
\node (M00) at (0,0) {$\sxu$};
\node (M01) at (2,0) {$\sxu$};
\node (M10) at (0,-2) {$\mathbb{R}$};
\node (M11) at (2,-2) {$\mathbb{R}$};
\draw[myarrow] (M00) to node[above]{$\sigma$} (M01);
\draw[myarrow] (M10) to node[below]{$f_{c}$} (M11);
\draw[myarrow] (M00) to node[left]{$\psi_{c}$} (M10);
\draw[myarrow] (M01) to node[right]{$\psi_{c}$} (M11);
\end{tikzpicture}
\end{center}

Furthermore, for every $\boldsymbol{\varepsilon} \in \sxu$, we have \[ \epsilon_{0} \psi_{c}(\boldsymbol{\varepsilon}) \in \left[ \sqrt{-\beta_{c} -c}, \beta_{c} \right] \, \text{,} \] for all $c \in (-\infty, -2]$, and the map $\zeta_{\boldsymbol{\varepsilon}} \colon (-\infty, -2] \rightarrow \mathbb{R}$ defined by \[ \zeta_{\boldsymbol{\varepsilon}}(c) = \psi_{c}(\boldsymbol{\varepsilon}) \] is continuous.
\end{theorem}

Before proving Theorem~\ref{theorem:dynashift}, observe that $c \leq -\beta_{c}$ for all $c \in (-\infty, -2]$, with equality if and only if $c = -2$. Consequently, for $c \in (-\infty, -2]$ and $\epsilon = \pm 1$, the partial inverse $g_{c}^{\epsilon} \colon [c, +\infty) \rightarrow \mathbb{R}$ of $f_{c}$ given by \[ g_{c}^{\epsilon}(z) = \epsilon \sqrt{z -c} \] is well defined on $\left[ -\beta_{c}, \beta_{c} \right]$, and we have \[ g_{c}^{\epsilon}\left( \left[ -\beta_{c}, \beta_{c} \right] \right) = \left[ \epsilon \sqrt{-\beta_{c} -c}, \epsilon \beta_{c} \right] \subset \left[ -\beta_{c}, \beta_{c} \right] \, \text{.} \]

\begin{lemma}
\label{lemma:dynashift}
For all $c \in (-\infty, -2]$ and all $\boldsymbol{\varepsilon} = \left( \epsilon_{0}, \dotsc, \epsilon_{p -1} \right) \in \lbrace -1, 1 \rbrace^{p}$, with $p \geq 1$, the map $g_{c}^{\boldsymbol{\varepsilon}} \colon \left[ -\beta_{c}, \beta_{c} \right] \rightarrow \left[ -\beta_{c}, \beta_{c} \right]$ defined by \[ g_{c}^{\boldsymbol{\varepsilon}}(z) = g_{c}^{\epsilon_{0}} \circ \dotsb \circ g_{c}^{\epsilon_{p -1}}(z) \] has a unique fixed point $\mathfrak{z}_{\boldsymbol{\varepsilon}}(c)$.

Moreover, for every finite sequence $\boldsymbol{\varepsilon}$ of $\pm 1$, the map $c \mapsto \mathfrak{z}_{\boldsymbol{\varepsilon}}(c)$ is continuous.
\end{lemma}

\begin{claim}
\label{claim:dynashift}
If $c \in (-\infty, -2]$, $\boldsymbol{\varepsilon} \in \lbrace -1, 1 \rbrace^{p}$, with $p \geq 1$, and $\mathfrak{z}$ is a fixed point of $g_{c}^{\boldsymbol{\varepsilon}}$, then $\mathfrak{z} \in \xv{c}{0}{p}$ and $\epsilon_{j} f_{c}^{\circ j}(\mathfrak{z}) > 0$ for all $j \in \lbrace 0, \dotsc, p -1 \rbrace$.
\end{claim}

\begin{proof}[Proof of Claim~\ref{claim:dynashift}]
We have $f_{c}^{\circ p}(\mathfrak{z}) = \mathfrak{z}$ and the set $\xv{c}{0}{p}$ is invariant under $f_{c}$. Therefore, for all $j \in \lbrace 0, \dotsc, p -1 \rbrace$, we have \[ f_{c}^{\circ j}(\mathfrak{z}) = g_{c}^{\epsilon_{j}} \circ \dotsb \circ g_{c}^{\epsilon_{p -1}}(\mathfrak{z}) \in g_{c}^{\epsilon_{j}}\left( \left[ -\beta_{c}, \beta_{c} \right] \right) \cap \xv{c}{0}{p} \, \text{,} \] which yields \[ \epsilon_{j} f_{c}^{\circ j}(\mathfrak{z}) \in \left( \sqrt{-\beta_{c} -c}, \beta_{c} \right] \subset \mathbb{R}_{> 0} \] since $\epsilon_{j} \sqrt{-\beta_{c} -c}$ is preperiodic for $f_{c}$ with preperiod $2$ and period $1$.
\end{proof}

\begin{proof}[Proof of Lemma~\ref{lemma:dynashift}]
Fix $c \in (-\infty, -2]$ and $p \geq 1$. For every $\boldsymbol{\varepsilon} \in \lbrace -1, 1 \rbrace^{p}$, the map $g_{c}^{\boldsymbol{\varepsilon}}$ has a fixed point $\mathfrak{z}_{\boldsymbol{\varepsilon}}(c)$ by the intermediate value theorem. Now, note that $\mathfrak{z}_{\boldsymbol{\varepsilon}}(c)$ is not a fixed point of $g_{c}^{\boldsymbol{\varepsilon}^{\prime}}$ whenever $\boldsymbol{\varepsilon} \neq \boldsymbol{\varepsilon}^{\prime} \in \lbrace -1, 1 \rbrace^{p}$ by Claim~\ref{claim:dynashift}. Therefore, the points $\mathfrak{z}_{\boldsymbol{\varepsilon}}(c)$, with $\boldsymbol{\varepsilon} \in \lbrace -1, 1 \rbrace^{p}$, are pairwise distinct, and, since $\xv{c}{0}{p}$ contains at most $2^{p}$ elements, it follows that \[ \xv{c}{0}{p} = \left\lbrace \mathfrak{z}_{\boldsymbol{\varepsilon}}(c) : \boldsymbol{\varepsilon} \in \lbrace -1, 1 \rbrace^{p} \right\rbrace \, \text{.} \] Thus, for every $\boldsymbol{\varepsilon} \in \lbrace -1, 1 \rbrace^{p}$, $\mathfrak{z}_{\boldsymbol{\varepsilon}}(c)$ is the unique fixed point of the map $g_{c}^{\boldsymbol{\varepsilon}}$.

Now, fix $p \geq 1$, $\boldsymbol{\varepsilon} = \left( \epsilon_{0}, \dotsc, \epsilon_{p -1} \right) \in \lbrace -1, 1 \rbrace^{p}$ and $c \in (-\infty, -2]$. It remains to verify that the map $c^{\prime} \mapsto \mathfrak{z}_{\boldsymbol{\varepsilon}}\left( c^{\prime} \right)$ is continuous at $c$. For each $c^{\prime} \in (-\infty, -2]$, choose $\boldsymbol{\varepsilon}_{c^{\prime}} \in \lbrace -1, 1 \rbrace^{p}$ such that $\left\lvert \mathfrak{z}_{\boldsymbol{\varepsilon}}(c) -\mathfrak{z}_{\boldsymbol{\varepsilon}_{c^{\prime}}}\left( c^{\prime} \right) \right\rvert$ is minimal. Then we have \[ \left\lvert \mathfrak{z}_{\boldsymbol{\varepsilon}}(c) -\mathfrak{z}_{\boldsymbol{\varepsilon}_{c^{\prime}}}\left( c^{\prime} \right) \right\rvert \leq \left( \prod_{\boldsymbol{\varepsilon}^{\prime} \in \lbrace -1, 1 \rbrace^{p}} \left\lvert \mathfrak{z}_{\boldsymbol{\varepsilon}}(c) -\mathfrak{z}_{\boldsymbol{\varepsilon}^{\prime}}\left( c^{\prime} \right) \right\rvert \right)^{\frac{1}{2^{p}}} = \left\lvert F_{p}\left( c^{\prime}, \mathfrak{z}_{\boldsymbol{\varepsilon}}(c) \right) -\mathfrak{z}_{\boldsymbol{\varepsilon}}(c) \right\rvert^{\frac{1}{2^{p}}} \] for all $c^{\prime} \in (-\infty, -2]$, and so $\mathfrak{z}_{\boldsymbol{\varepsilon}_{c^{\prime}}}\left( c^{\prime} \right)$ tends to $\mathfrak{z}_{\boldsymbol{\varepsilon}}(c)$ as $c^{\prime}$ approaches $c$. By Claim~\ref{claim:dynashift}, it follows that, whenever $c^{\prime}$ is close enough to $c$, we have $\epsilon_{j} f_{c^{\prime}}^{\circ j}\left( \mathfrak{z}_{\boldsymbol{\varepsilon}_{c^{\prime}}}\left( c^{\prime} \right) \right) > 0$ for all $j \in \lbrace 0, \dotsc, p -1 \rbrace$, which yields $\boldsymbol{\varepsilon}_{c^{\prime}} = \boldsymbol{\varepsilon}$. Thus, the limit of $\mathfrak{z}_{\boldsymbol{\varepsilon}}\left( c^{\prime} \right)$ as $c^{\prime}$ approaches $c$ is $\mathfrak{z}_{\boldsymbol{\varepsilon}}(c)$, and the lemma is proved.
\end{proof}

We may now deduce Theorem~\ref{theorem:dynashift} from Lemma~\ref{lemma:dynashift}.

\begin{proof}[Proof of Theorem~\ref{theorem:dynashift}]
Fix $c \in (-\infty, -2]$. Assume that $\psi_{c} \colon \sxu \rightarrow \mathbb{R}$ is a map that satisfies $f_{c} \circ \psi_{c} = \psi_{c} \circ \sigma$ and $\epsilon_{0} \psi_{c}(\boldsymbol{\varepsilon}) \geq 0$ for all $\boldsymbol{\varepsilon} \in \sxu$. Then, for all $\boldsymbol{\varepsilon} \in \sxu$ and all $n \geq 0$, we have \[ \psi_{c}(\boldsymbol{\varepsilon}) = g_{c}^{\epsilon_{0}} \circ \dotsb \circ g_{c}^{\epsilon_{n}}\left( \psi_{c}\left( \sigma^{\circ (n +1)}(\boldsymbol{\varepsilon}) \right) \right) \, \text{.} \] It follows that, if $\boldsymbol{\varepsilon}$ is a periodic sign sequence with period $p \geq 1$, then $\psi_{c}(\boldsymbol{\varepsilon})$ is a fixed point of the map $g_{c}^{\boldsymbol{\varepsilon}_{\boldsymbol{p}}}$, where $\boldsymbol{\varepsilon}_{\boldsymbol{p}} = \left( \epsilon_{0}, \dotsc, \epsilon_{p -1} \right) \in \lbrace -1, 1 \rbrace^{p}$, and hence $\psi_{c}(\boldsymbol{\varepsilon}) = \mathfrak{z}_{\boldsymbol{\varepsilon}_{\boldsymbol{p}}}(c)$. Therefore, for every $\boldsymbol{\varepsilon} \in \sxu$ with preperiod $k \geq 0$ and period $p \geq 1$, we have $\psi_{c}(\boldsymbol{\varepsilon}) = g_{c}^{\boldsymbol{\varepsilon}_{\boldsymbol{pp}}}\left( \mathfrak{z}_{\boldsymbol{\varepsilon}_{\boldsymbol{p}}}(c) \right)$, where $\boldsymbol{\varepsilon}_{\boldsymbol{pp}} = \left( \epsilon_{0}, \dotsc, \epsilon_{k -1} \right) \in \lbrace -1, 1 \rbrace^{k}$ and $\boldsymbol{\varepsilon}_{\boldsymbol{p}} = \left( \epsilon_{k}, \dotsc, \epsilon_{k +p -1} \right) \in \lbrace -1, 1 \rbrace^{p}$, adopting the convention that $g_{c}^{\varnothing}$ denotes the identity map of $\left[ -\beta_{c}, \beta_{c} \right]$. In particular, there is at most one map $\psi_{c} \colon \sxu \rightarrow \mathbb{R}$ that satisfies the conditions above.

For $\boldsymbol{\varepsilon} = \left( \epsilon_{n} \right)_{n \geq 0}$ a preperiodic sign sequence with preperiod $k \geq 0$ and period $p \geq 1$, define $\boldsymbol{\varepsilon}_{\boldsymbol{pp}} = \left( \epsilon_{0}, \dotsc, \epsilon_{k -1} \right) \in \lbrace -1, 1 \rbrace^{k}$, $\boldsymbol{\varepsilon}_{\boldsymbol{p}} = \left( \epsilon_{k}, \dotsc, \epsilon_{k +p -1} \right) \in \lbrace -1, 1 \rbrace^{p}$ and $\psi_{c}(\boldsymbol{\varepsilon}) = g_{c}^{\boldsymbol{\varepsilon}_{\boldsymbol{pp}}}\left( \mathfrak{z}_{\boldsymbol{\varepsilon}_{\boldsymbol{p}}}(c) \right)$. If $\boldsymbol{\varepsilon}$ is a periodic sign sequence with period $p \geq 1$, then $f_{c} \circ \psi_{c}(\boldsymbol{\varepsilon})$ is a fixed point of the map $g_{c}^{\sigma(\boldsymbol{\varepsilon})_{\boldsymbol{p}}}$ since $\sigma(\boldsymbol{\varepsilon})_{\boldsymbol{p}} = \left( \epsilon_{1}, \dotsc, \epsilon_{p -1}, \epsilon_{0} \right)$, and hence $f_{c} \circ \psi_{c}(\boldsymbol{\varepsilon}) = \psi_{c} \circ \sigma(\boldsymbol{\varepsilon})$. Similarly, if $\boldsymbol{\varepsilon} \in \sxu$ has preperiod $k \geq 1$ and period $p \geq 1$, then $f_{c} \circ \psi_{c}(\boldsymbol{\varepsilon}) = \psi_{c} \circ \sigma(\boldsymbol{\varepsilon})$ since $\sigma(\boldsymbol{\varepsilon})_{\boldsymbol{pp}} = \left( \epsilon_{1}, \dotsc, \epsilon_{k -1} \right)$ and $\sigma(\boldsymbol{\varepsilon})_{\boldsymbol{p}} = \boldsymbol{\varepsilon}_{\boldsymbol{p}}$. Moreover, for all $\boldsymbol{\varepsilon} \in \sxu$, we have $\psi_{c}(\boldsymbol{\varepsilon}) \in g_{c}^{\epsilon_{0}}\left( \left[ -\beta_{c}, \beta_{c} \right] \right)$, which yields \[ \epsilon_{0} \psi_{c}(\boldsymbol{\varepsilon}) \in \left[ \sqrt{-\beta_{c} -c}, \beta_{c} \right] \subset \mathbb{R}_{\geq 0} \, \text{.} \] Thus, the map $\psi_{c} \colon \sxu \rightarrow \mathbb{R}$ so defined has the required properties.

Furthermore, for every $\boldsymbol{\varepsilon} \in \sxu$, the map $\zeta_{\boldsymbol{\varepsilon}} \colon c \mapsto \psi_{c}(\boldsymbol{\varepsilon})$ is clearly continuous.
\end{proof}

\begin{remark}
Observe that, if $c \in (-\infty, -2]$ and $\boldsymbol{\varepsilon}, \boldsymbol{\varepsilon}^{\prime} \in \sxu$ satisfy $\epsilon_{0} = -\epsilon_{0}^{\prime}$ and $\sigma(\boldsymbol{\varepsilon}) = \sigma\left( \boldsymbol{\varepsilon}^{\prime} \right)$, then $\psi_{c}(\boldsymbol{\varepsilon}) = -\psi_{c}\left( \boldsymbol{\varepsilon}^{\prime} \right)$.
\end{remark}

Note that the proof of Theorem~\ref{theorem:dynashift} provides explicit formulas for the maps $\zeta_{\boldsymbol{\varepsilon}}$ with $\boldsymbol{\varepsilon} \in \sxv{k}{1}$ and $k \geq 0$.

\begin{example}
Suppose that $\epsilon = \pm 1$. Then 
\begin{itemize}
\item for $\boldsymbol{\varepsilon} \in \sxv{1}{1}$ given by $\epsilon_{0} = \epsilon$ and $\epsilon_{1} = -1$, we have \[ \zeta_{\boldsymbol{\varepsilon}} \colon c \mapsto \psi_{c}(\boldsymbol{\varepsilon}) = -\epsilon \alpha_{c} \, \text{;} \]
\item for $\boldsymbol{\varepsilon} \in \sxv{1}{1}$ given by $\epsilon_{0} = \epsilon$ and $\epsilon_{1} = 1$, we have \[ \zeta_{\boldsymbol{\varepsilon}} \colon c \mapsto \psi_{c}(\boldsymbol{\varepsilon}) = \epsilon \beta_{c} \, \text{;} \]
\item for $\boldsymbol{\varepsilon} \in \sxv{2}{1}$ given by $\epsilon_{0} = \epsilon$, $\epsilon_{1} = 1$ and $\epsilon_{2} = -1$, we have \[ \zeta_{\boldsymbol{\varepsilon}} \colon c \mapsto \psi_{c}(\boldsymbol{\varepsilon}) = \epsilon \sqrt{-\alpha_{c} -c} \, \text{;} \]
\item for $\boldsymbol{\varepsilon} \in \sxv{2}{1}$ given by $\epsilon_{0} = \epsilon$, $\epsilon_{1} = -1$ and $\epsilon_{2} = 1$, we have \[ \zeta_{\boldsymbol{\varepsilon}} \colon c \mapsto \psi_{c}(\boldsymbol{\varepsilon}) = \epsilon \sqrt{-\beta_{c} -c} \, \text{.} \]
\end{itemize}
\end{example}

\begin{proposition}
\label{proposition:dynashift}
Assume that $c \in (-\infty, -2]$. Then we have \[ \xv{c}{k}{p} = \psi_{c}\left( \sxv{k}{p} \right) \subset \left[ -\beta_{c}, \beta_{c} \right] \] for all $k \geq 0$ and $p \geq 1$ (see Figure~\ref{figure:dynashift}).

Furthermore, if $c \in (-\infty, -2)$, then the map $\psi_{c} \colon \sxu \rightarrow \mathbb{R}$ is injective.
\end{proposition}

\begin{figure}
\fbox{\includegraphics{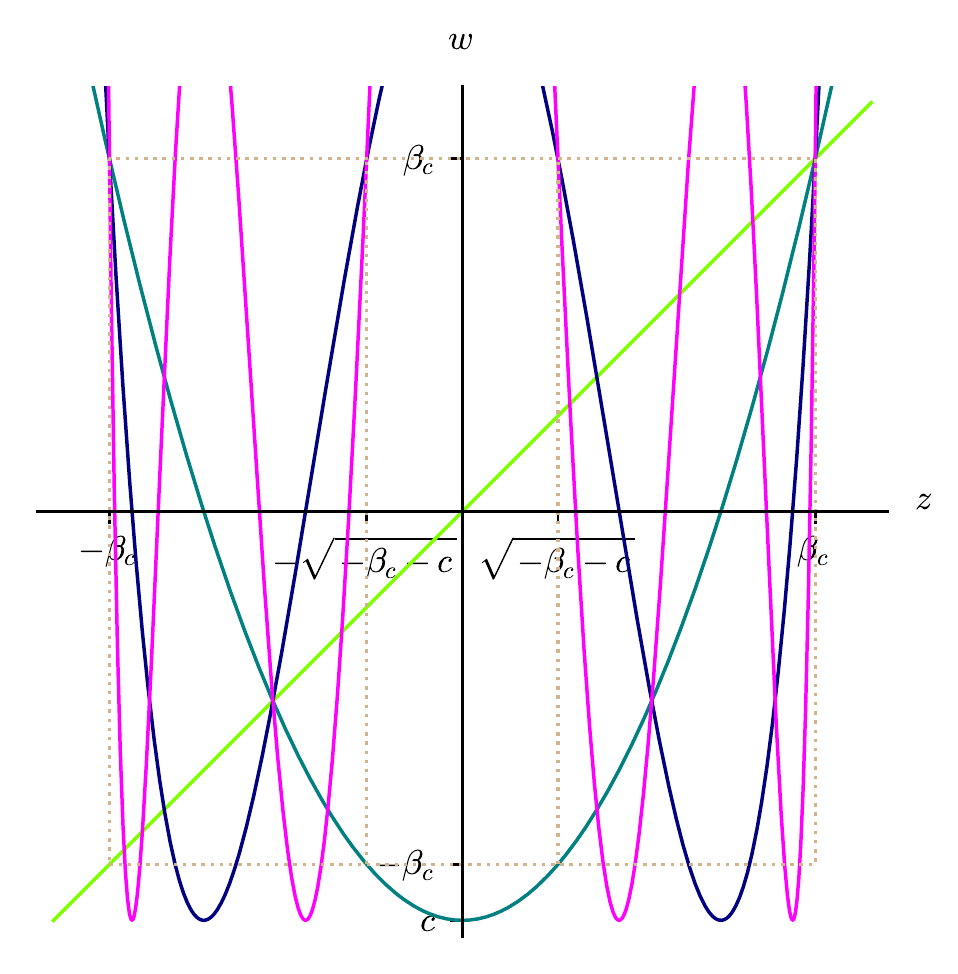}}
\caption{Graphs of the maps $z \mapsto F_{n}(c, z)$, with $n \in \lbrace 0, \dotsc, 3 \rbrace$, when $c \in (-\infty, -2]$.}
\label{figure:dynashift}
\end{figure}

\begin{proof}
For all $n \geq 0$, we have $f_{c}^{\circ n} \circ \psi_{c} = \psi_{c} \circ \sigma^{\circ n}$. Consequently, $\psi_{c}\left( \sxv{k}{p} \right) \subset \xv{c}{k}{p}$ for all $k \geq 0$ and $p \geq 1$.

Now, suppose that $c \in (-\infty, -2)$. Then, for all $\boldsymbol{\varepsilon} \in \sxu$ and all $n \geq 0$, we have \[ \epsilon_{n} f_{c}^{\circ n}\left( \psi_{c}(\boldsymbol{\varepsilon}) \right) \in \left[ \sqrt{-\beta_{c} -c}, \beta_{c} \right] \subset \mathbb{R}_{> 0} \, \text{.} \] Therefore, the map $\psi_{c}$ is injective, and, since $\xv{c}{k}{p}$ contains at most $2^{k +p}$ elements, it follows that $\psi_{c}\left( \sxv{k}{p} \right) = \xv{c}{k}{p}$, for all $k \geq 0$ and $p \geq 1$.

It remains to prove that $\xv{-2}{k}{p} \subset \psi_{-2}\left( \sxv{k}{p} \right)$ for all $k \geq 0$ and $p \geq 1$. Fix $k \geq 0$ and $p \geq 1$, and suppose that $z \in \xv{-2}{k}{p}$. Then, for all $c \in (-\infty, -2)$, we have \[ \min_{\boldsymbol{\varepsilon} \in \sxv{k}{p}} \left\lvert z -\psi_{c}(\boldsymbol{\varepsilon}) \right\rvert \leq \left( \prod_{\boldsymbol{\varepsilon} \in \sxv{k}{p}} \left\lvert z -\psi_{c}(\boldsymbol{\varepsilon}) \right\rvert \right)^{\frac{1}{2^{k +p}}} = \left\lvert F_{k +p}(c, z) -F_{k}(c, z) \right\rvert^{\frac{1}{2^{k +p}}} \, \text{.} \] As the maps $\zeta_{\boldsymbol{\varepsilon}}$, with $\boldsymbol{\varepsilon} \in \sxv{k}{p}$, are continuous at $-2$, it follows that $z \in \psi_{-2}\left( \sxv{k}{p} \right)$. Thus, the proposition is proved.
\end{proof}

\begin{remark}
Applying Montel's theorem, it follows from Proposition~\ref{proposition:dynashift} that, for every $c \in (-\infty, -2]$, the filled-in Julia set of $f_{c}$~-- that is, the set of points $z \in \mathbb{C}$ that have bounded forward orbit under $f_{c}$~-- is also contained in $\left[ -\beta_{c}, \beta_{c} \right]$.
\end{remark}

Note that the map $\psi_{-2}$ is not injective. More precisely, we have the following:

\begin{proposition}
\label{proposition:dynapart}
For all $\boldsymbol{\varepsilon} \neq \boldsymbol{\varepsilon}^{\prime} \in \sxu$, $\psi_{-2}(\boldsymbol{\varepsilon}) = \psi_{-2}\left( \boldsymbol{\varepsilon}^{\prime} \right)$ if and only if there exists an integer $k \geq 2$ such that $\boldsymbol{\varepsilon}, \boldsymbol{\varepsilon}^{\prime} \in \sxv{k}{1}$, $\epsilon_{j} = \epsilon_{j}^{\prime}$ for all $j \in \lbrace 0, \dotsc, k -3 \rbrace$, $\epsilon_{k -2} = -\epsilon_{k -2}^{\prime}$, $\epsilon_{k -1} = \epsilon_{k -1}^{\prime} = -1$ and $\epsilon_{k} = \epsilon_{k}^{\prime} = 1$.
\end{proposition}

\begin{proof}
Suppose that $\boldsymbol{\varepsilon} \neq \boldsymbol{\varepsilon}^{\prime} \in \sxu$ satisfy $\psi_{-2}(\boldsymbol{\varepsilon}) = \psi_{-2}\left( \boldsymbol{\varepsilon}^{\prime} \right)$. Then, for all $n \geq 0$, we have \[ \epsilon_{n} f_{-2}^{\circ n}\left( \psi_{-2}(\boldsymbol{\varepsilon}) \right) \geq 0 \quad \text{and} \quad \epsilon_{n}^{\prime} f_{-2}^{\circ n}\left( \psi_{-2}(\boldsymbol{\varepsilon}) \right) \geq 0 \, \text{.} \] Since $\boldsymbol{\varepsilon} \neq \boldsymbol{\varepsilon}^{\prime}$, it follows that there is an integer $k \geq 0$, which we may assume minimal, such that $f_{-2}^{\circ k}\left( \psi_{-2}(\boldsymbol{\varepsilon}) \right) = 0$. For all $j \in \lbrace 0, \dotsc, k -1 \rbrace$, the inequalities above are strict, and hence $\epsilon_{j} = \epsilon_{j}^{\prime}$. Moreover, we have $f_{-2}^{\circ (k +1)}\left( \psi_{-2}(\boldsymbol{\varepsilon}) \right) = -2$ and $f_{-2}^{\circ n}\left( \psi_{-2}(\boldsymbol{\varepsilon}) \right) = 2$ for all $n \geq k +2$, which yields $\epsilon_{k +1} = \epsilon_{k +1}^{\prime} = -1$ and $\epsilon_{n} = \epsilon_{n}^{\prime} = 1$ for all $n \geq k +2$. Thus, the sign sequences $\boldsymbol{\varepsilon}$ and $\boldsymbol{\varepsilon}^{\prime}$ have the desired form.

Conversely, observe that, for $\boldsymbol{\varepsilon} \in \sxv{2}{1}$ with $\epsilon_{1} = -1$ and $\epsilon_{2} = 1$, we have \[ \psi_{-2}(\boldsymbol{\varepsilon}) = \epsilon_{0} \sqrt{-\beta_{-2} -(-2)} = 0 \, \text{.} \] Therefore, if $k \geq 2$ and $\boldsymbol{\varepsilon} \in \sxv{k}{1}$ satisfies $\epsilon_{k -1} = -1$ and $\epsilon_{k} = 1$, then \[ \psi_{-2}(\boldsymbol{\varepsilon}) = g_{-2}^{\left( \epsilon_{0}, \dotsc, \epsilon_{k -3} \right)}\left( \psi_{-2}\left( \sigma^{\circ (k -2)}(\boldsymbol{\varepsilon}) \right) \right) = g_{-2}^{\left( \epsilon_{0}, \dotsc, \epsilon_{k -3} \right)}(0) \] does not depend on $\epsilon_{k -2}$. This completes the proof of the proposition.
\end{proof}

\begin{remark}
It follows from Proposition~\ref{proposition:dynashift} and Proposition~\ref{proposition:dynapart} that, for all $k \geq 0$ and $p \geq 1$, the set $\xv{-2}{k}{p}$ contains exactly $2^{p}$ elements if $k = 0$ and $2^{k +p} -2^{k -1} +1$ elements if $k \geq 1$.
\end{remark}

\begin{remark}
Note that we can actually describe the map $\psi_{-2} \colon \sxu \rightarrow \mathbb{R}$ explicitly. For $\boldsymbol{\varepsilon} \in \sxu$, define the sequence $\left( \delta_{n}(\boldsymbol{\varepsilon}) \right)_{n \geq 0} \in \lbrace 0, 1 \rbrace^{\mathbb{Z}_{\geq 0}}$ by \[ \delta_{n}(\boldsymbol{\varepsilon}) = \begin{cases} \delta_{n -1}(\boldsymbol{\varepsilon}) & \text{if } \epsilon_{n} = 1\\ 1 -\delta_{n -1}(\boldsymbol{\varepsilon}) & \text{if } \epsilon_{n} = -1 \end{cases} \, \text{,} \] where $\delta_{-1}(\boldsymbol{\varepsilon}) = 0$ by convention. Then the map $\psi_{-2} \colon \sxu \rightarrow \mathbb{R}$ is given by \[ \psi_{-2}(\boldsymbol{\varepsilon}) = 2 \cos\left( \pi \sum_{n = 0}^{+\infty} \frac{\delta_{n}(\boldsymbol{\varepsilon})}{2^{n +1}} \right) \, \text{.} \]
\end{remark}

\section{Back to the parameter space}
\label{section:paraspace}

We shall now exploit the statements given in Section~\ref{section:dynaspace} to get results concerning the parameter space.

\begin{remark}
By definition, for every point $a \in \mathbb{C}$ and every parameter $c \in \mathbb{C}$, $c \in \su{a}$ if and only if $a \in \xu{c}$ and, for all $k \geq 0$ and $p \geq 1$, $c \in \sv{a}{k}{p}$ if and only if $a \in \xv{c}{k}{p}$.
\end{remark}

\begin{proposition}
\label{proposition:paraloc}
For every $a \in \mathbb{C}$, we have \[ \su{a} \subset \left\lbrace c \in \mathbb{C} : \lvert c \rvert \leq R_{a} \right\rbrace \, \text{,} \] where $R_{a} = \lvert a \rvert^{2} +\sqrt{\lvert a \rvert^{2} +1} +1$.
\end{proposition}

\begin{proof}
Suppose that $c \in \su{a}$. Then, by Proposition~\ref{proposition:dynaloc}, we have \[ \lvert c \rvert -\lvert a \rvert^{2} \leq \left\lvert f_{c}(a) \right\rvert \leq \rho_{c} \, \text{,} \] and hence $\varphi\left( \lvert c \rvert \right) \leq \lvert a \rvert^{2}$, where $\varphi \colon \mathbb{R}_{\geq 0} \rightarrow \mathbb{R}$ is given by \[ \varphi(x) = x -\frac{1 +\sqrt{1 +4 x}}{2} \, \text{.} \] The map $\varphi$ is strictly increasing and satisfies $\varphi\left( R_{a} \right) = \lvert a \rvert^{2}$. Thus, the proposition is proved.
\end{proof}

Now, let us give a more extensive description of $\su{a}$ when $a \in (-\infty, -2] \cup [2, +\infty)$.

Given $\epsilon = \pm 1$, let $\sxvs{\epsilon}{k}{p}$~-- with $k \geq 0$ and $p \geq 1$~-- be the set defined by \[ \sxvs{\epsilon}{k}{p} = \left\lbrace \boldsymbol{\varepsilon} = \left( \epsilon_{n} \right)_{n \geq 0} \in \sxv{k}{p} : \epsilon_{0} = \epsilon \right\rbrace \, \text{,} \] and let $\sxus{\epsilon}$ be the set defined by \[ \sxus{\epsilon} = \bigcup_{k \geq 0, \, p \geq 1} \sxvs{\epsilon}{k}{p} = \left\lbrace \boldsymbol{\varepsilon} \in \sxu : \epsilon_{0} = \epsilon \right\rbrace \, \text{.} \] For all $k \geq 0$ and $p \geq 1$, the set $\sxvs{\epsilon}{k}{p}$ contains exactly $2^{k +p -1}$ elements~-- each of them being completely determined by the choice of its terms with index in $\lbrace 1, \dotsc, k +p -1 \rbrace$.

Suppose that $a \in (-\infty, -2] \cup [2, +\infty)$. Then 
\begin{itemize}
\item for $\boldsymbol{\varepsilon} \in \sxvs{\sgn(a)}{2}{1}$ given by $\epsilon_{1} = -1$ and $\epsilon_{2} = 1$, the map \[ \sgn(a) \zeta_{\boldsymbol{\varepsilon}} \colon c \mapsto \sqrt{-\beta_{c} -c} \] is strictly decreasing on $(-\infty, -2]$ and we have $\zeta_{\boldsymbol{\varepsilon}}\left( c_{a}^{-} \right) = a$, where $c_{a}^{-}$ is the parameter defined by \[ c_{a}^{-} = -a^{2} -\sqrt{a^{2} +1} -1 \in \sv{a}{2}{1} \, \text{;} \]
\item for $\boldsymbol{\varepsilon} \in \sxvs{\sgn(a)}{1}{1}$ given by $\epsilon_{1} = 1$, the map \[ \sgn(a) \zeta_{\boldsymbol{\varepsilon}} \colon c \mapsto \beta_{c} \] is strictly decreasing on $(-\infty, -2]$ and we have $\zeta_{\boldsymbol{\varepsilon}}\left( c_{a}^{+} \right) = a$, where $c_{a}^{+}$ is the parameter defined by \[ c_{a}^{+} = -a^{2} +\lvert a \rvert \in \sv{a}{1}{1} \, \text{.} \]
\end{itemize}

\begin{remark}
Note that, for every $a \in \mathbb{C}$ with $\lvert a \rvert \geq 2$, we have $R_{a} = -c_{\lvert a \rvert}^{-}$.
\end{remark}

\begin{theorem}
\label{theorem:parashift}
Assume that $a \in (-\infty, -2] \cup [2, +\infty)$. Then there is a unique map \[ \gamma_{a} \colon \sxus{\sgn(a)} \rightarrow (-\infty, -2] \] that satisfies $\zeta_{\boldsymbol{\varepsilon}}\left( \gamma_{a}(\boldsymbol{\varepsilon}) \right) = a$ for all $\boldsymbol{\varepsilon} \in \sxus{\sgn(a)}$ (see Figure~\ref{figure:parashift1}).

Furthermore, we have \[ \sv{a}{k}{p} = \gamma_{a}\left( \sxvs{\sgn(a)}{k}{p} \right) \subset \left[ c_{a}^{-}, c_{a}^{+} \right] \, \text{,} \] for all $k \geq 0$ and $p \geq 1$, (see Figure~\ref{figure:parashift2}) and the map $\gamma_{a}$ is injective.
\end{theorem}

\begin{figure}
\fbox{\includegraphics{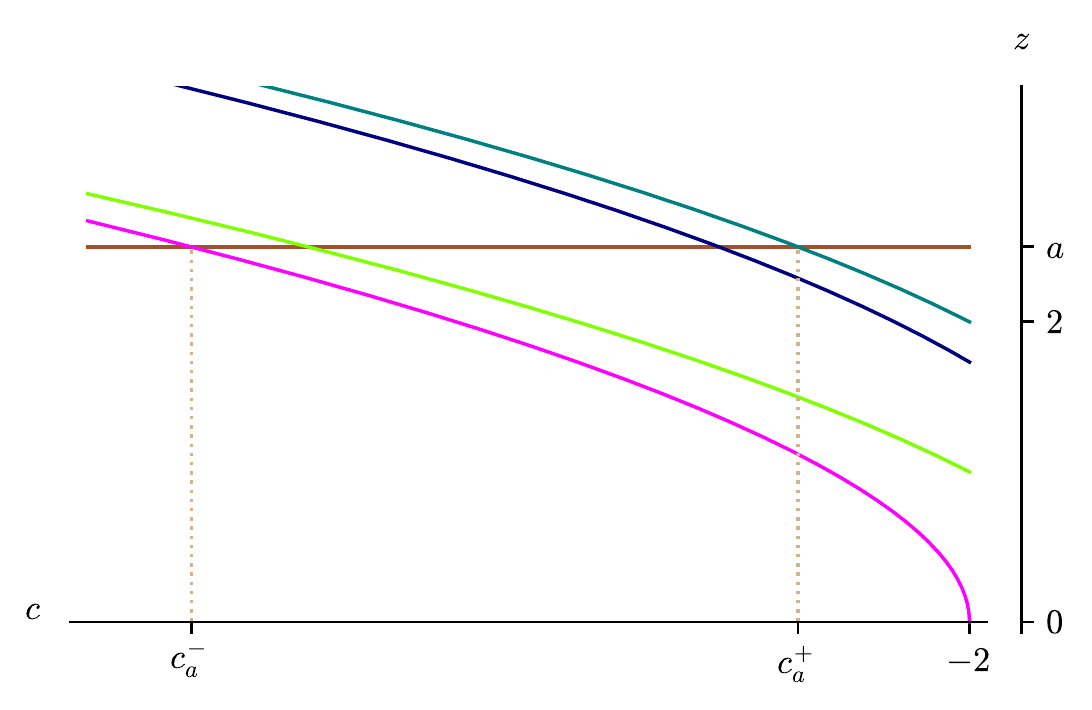}}
\caption{Graphs of the maps $\zeta_{\boldsymbol{\varepsilon}}$, with $\boldsymbol{\varepsilon} \in \sxvs{\sgn(a)}{2}{1}$, when $a \in [2, +\infty)$.}
\label{figure:parashift1}
\end{figure}

\begin{figure}
\fbox{\includegraphics{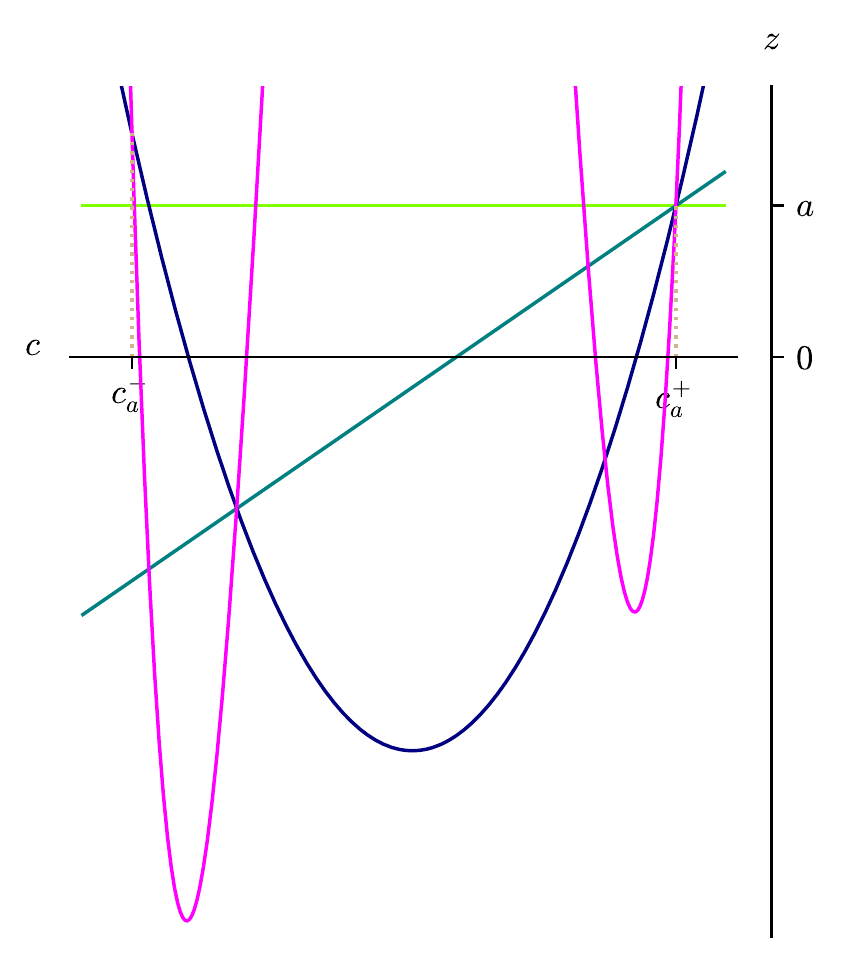}}
\caption{Graphs of the maps $c \mapsto F_{n}(c, a)$, with $n \in \lbrace 0, \dotsc, 3 \rbrace$, when $a \in [2, +\infty)$.}
\label{figure:parashift2}
\end{figure}

\begin{claim}
\label{claim:parashift}
If $a \in (-\infty, -2] \cup [2, +\infty)$ and $\gamma \in (-\infty, -2]$, then $a$ has at most one preimage under $\psi_{\gamma}$.
\end{claim}

\begin{proof}[Proof of Claim~\ref{claim:parashift}]
If $\gamma \in (-\infty, -2)$, then the map $\psi_{\gamma}$ is injective.

If $\gamma = -2$ and $\boldsymbol{\varepsilon} \in \sxu$ satisfies $\psi_{\gamma}(\boldsymbol{\varepsilon}) = a$, then we have \[ 2 \leq \lvert a \rvert = \left\lvert \psi_{-2}(\boldsymbol{\varepsilon}) \right\rvert \leq \beta_{-2} = 2 \, \text{,} \] so $\psi_{-2}(\boldsymbol{\varepsilon}) = \sgn(a) \beta_{-2}$, and, by Proposition~\ref{proposition:dynapart}, it follows that $\boldsymbol{\varepsilon}$ is the sign sequence in $\sxvs{\sgn(a)}{1}{1}$ given by $\epsilon_{1} = 1$. Thus, the claim is proved.
\end{proof}

\begin{proof}[Proof of Theorem~\ref{theorem:parashift}]
For every $\boldsymbol{\varepsilon} \in \sxus{\sgn(a)}$, we have \[ \sgn(a) \zeta_{\boldsymbol{\varepsilon}}\left( c_{a}^{-} \right) \geq \sqrt{-\beta_{c_{a}^{-}} -c_{a}^{-}} = \lvert a \rvert \quad \text{and} \quad \sgn(a) \zeta_{\boldsymbol{\varepsilon}}\left( c_{a}^{+} \right) \leq \beta_{c_{a}^{+}} = \lvert a \rvert \, \text{,} \] and hence, by the intermediate value theorem, there exists $\gamma_{a}(\boldsymbol{\varepsilon}) \in \left[ c_{a}^{-}, c_{a}^{+} \right]$ such that $\zeta_{\boldsymbol{\varepsilon}}\left( \gamma_{a}(\boldsymbol{\varepsilon}) \right) = a$. Now, note that, if $\boldsymbol{\varepsilon} \in \sxvs{\sgn(a)}{k}{p}$~-- with $k \geq 0$ and $p \geq 1$~-- and $\gamma \in (-\infty, -2]$ satisfy $\zeta_{\boldsymbol{\varepsilon}}(\gamma) = a$, then $\boldsymbol{\varepsilon}$ is a preimage of $a$ under $\psi_{\gamma}$, and in particular $\gamma \in \sv{a}{k}{p}$. Therefore, by Claim~\ref{claim:parashift}, the map $\gamma_{a}$ so defined is injective, and, as $\sv{a}{k}{p}$ contains at most $2^{k +p -1}$ elements, it follows that $\gamma_{a}\left( \sxvs{\sgn(a)}{k}{p} \right) = \sv{a}{k}{p}$, for all $k \geq 0$ and $p \geq 1$. Thus, for every $\boldsymbol{\varepsilon} \in \sxus{\sgn(a)}$, $\gamma_{a}(\boldsymbol{\varepsilon})$ is the unique parameter $\gamma \in (-\infty, -2]$ that satisfies $\zeta_{\boldsymbol{\varepsilon}}(\gamma) = a$. This completes the proof of the theorem.
\end{proof}

\begin{remark}
Applying Montel's theorem, it follows from Theorem~\ref{theorem:parashift} that, for every $a \in (-\infty, -2] \cup [2, +\infty)$, the set of parameters $c \in \mathbb{C}$ for which the point $a$ has bounded forward orbit under $f_{c}$ is also contained in the line segment $\left[ c_{a}^{-}, c_{a}^{+} \right]$.
\end{remark}

Note that, when $a$ is an integer, the set $\su{a}$ has the following arithmetic property:

\begin{proposition}
\label{proposition:parathmetic}
For every $a \in \mathbb{Z}$, the set $\su{a}$ is contained in the set of algebraic integers and is invariant under the action of $\Gal\left( \overline{\mathbb{Q}} / \mathbb{Q} \right)$.
\end{proposition}

\begin{proof}
For all $k \geq 0$ and $p \geq 1$, the polynomial $F_{k +p}(c, a) -F_{k}(c, a)$ is monic with integer coefficients since $a \in \mathbb{Z}$. Thus, the proposition is proved.
\end{proof}

We shall now prove Theorem~\ref{theorem:mainthm}, which we recall below.

\begin{maintheorem}
Assume that $a$ and $b$ are two integers with $\lvert b \rvert > \lvert a \rvert$. Then 
\begin{itemize}
\item either $a = 0$, $\lvert b \rvert = 1$ and $\su{a} \cap \su{b} = \lbrace -2, -1, 0 \rbrace$,
\item or $a = 0$, $\lvert b \rvert = 2$ and $\su{a} \cap \su{b} = \lbrace -2 \rbrace$,
\item or $\lvert a \rvert \geq 1$, $\lvert b \rvert = \lvert a \rvert +1$ and $\su{a} \cap \su{b} = \left\lbrace -a^{2} -\lvert a \rvert -1, -a^{2} -\lvert a \rvert \right\rbrace$,
\item or $\lvert b \rvert > \max\left\lbrace 2, \lvert a \rvert +1 \right\rbrace$ and $\su{a} \cap \su{b} = \varnothing$.
\end{itemize}
\end{maintheorem}

\begin{lemma}
\label{lemma:mainthm}
Assume that $m \in \mathbb{Z}$ and $c$ is an algebraic integer whose all Galois conjugates lie in the interval $(m -2, m]$. Then $c = m -1$ or $c = m$.
\end{lemma}

\begin{proof}[Proof of Lemma~\ref{lemma:mainthm}]
Set $\alpha = c -m +1$. Then $\alpha$ is an algebraic integer whose all Galois conjugates $\alpha_{1}, \dotsc, \alpha_{d}$ lie in the interval $(-1, 1]$. Therefore, we have \[ \prod_{j = 1}^{d} \alpha_{j} \in (-1, 1] \cap \mathbb{Z} = \lbrace 0, 1 \rbrace \, \text{,} \] and it follows that either $\alpha_{j} = 0$ for some $j \in \lbrace 1, \dotsc, d \rbrace$, which yields $\alpha = 0$, or $\alpha_{j} = 1$ for all $j \in \lbrace 1, \dotsc, d \rbrace$. Thus, either $c = m -1$ or $c = m$.
\end{proof}

\begin{proof}[Proof of Theorem~\ref{theorem:mainthm}]
For a proof of the case $a = 0$ and $\lvert b \rvert = 1$, we refer the reader to~\cite[Proposition~6]{B2018}.

Thus, we may assume that $\lvert b \rvert \geq 2$. By Proposition~\ref{proposition:paraloc}, Theorem~\ref{theorem:parashift} and Proposition~\ref{proposition:parathmetic}, the set $\su{a} \cap \su{b}$ is contained in the set of algebraic integers, is invariant under the action of $\Gal\left( \overline{\mathbb{Q}} / \mathbb{Q} \right)$ and satisfies \[ \su{a} \cap \su{b} \subset \left\lbrace c \in \mathbb{C} : \lvert c \rvert \leq R_{a} \right\rbrace \cap \left[ c_{b}^{-}, c_{b}^{+} \right] \, \text{.} \]

Suppose that $a = 0$. Then we have \[ c_{b}^{+} = -b^{2} +\lvert b \rvert \leq -2 = -R_{a} \, \text{,} \] with equality if and only if $\lvert b \rvert = 2$. Therefore, $\su{a} \cap \su{b} \subset \lbrace -2 \rbrace$ if $\lvert b \rvert = 2$ and $\su{a} \cap \su{b} = \varnothing$ otherwise. Conversely, observe that $-2 \in \sv{a}{2}{1} \cap \sv{b}{1}{1}$ when $\lvert b \rvert = 2$.

Now, suppose that $\lvert a \rvert \geq 1$. Then we have \[ c_{b}^{+} -2 < -R_{a} = -a^{2} -\sqrt{a^{2} +1} -1 < -a^{2} -\lvert a \rvert = c_{b}^{+} \quad \text{if} \quad \lvert b \rvert = \lvert a \rvert +1 \] and \[ c_{b}^{+} = -b^{2} +\lvert b \rvert < -a^{2} -\sqrt{a^{2} +1} -1 = -R_{a} \quad \text{if} \quad \lvert b \rvert \geq \lvert a \rvert +2 \, \text{.} \] Therefore, $\su{a} \cap \su{b} \subset \left\lbrace -a^{2} -\lvert a \rvert -1, -a^{2} -\lvert a \rvert \right\rbrace$ if $\lvert b \rvert = \lvert a \rvert +1$ by Lemma~\ref{lemma:mainthm} and $\su{a} \cap \su{b} = \varnothing$ otherwise. Conversely, observe that $-a^{2} -\lvert a \rvert -1 \in \sv{a}{1}{2} \cap \sv{b}{1}{2}$ and $-a^{2} -\lvert a \rvert \in \sv{a}{1}{1} \cap \sv{b}{1}{1}$ when $\lvert b \rvert = \lvert a \rvert +1$. Thus, the theorem is proved.
\end{proof}

\providecommand{\bysame}{\leavevmode\hbox to3em{\hrulefill}\thinspace}
\providecommand{\MR}{\relax\ifhmode\unskip\space\fi MR }
\providecommand{\MRhref}[2]{%
\href{http://www.ams.org/mathscinet-getitem?mr=#1}{#2}
}
\providecommand{\href}[2]{#2}


\begin{thebibliography}{Buf18}

\bibitem[BD11]{BDM2011}
Matthew Baker and Laura DeMarco, \emph{Preperiodic points and unlikely
intersections}, Duke Math. J. \textbf{159} (2011), no.~1, 1--29. \MR{2817647}

\bibitem[Buf18]{B2018}
Xavier Buff, \emph{On postcritically finite unicritical polynomials}, New York
J. Math. \textbf{24} (2018), 1111--1122. \MR{3890968}

\end{thebibliography}
\end{document}